\documentclass[11pt,bezier]{article}
\usepackage{amsmath,amssymb,amsfonts,euscript}
\usepackage{algorithmic, algorithm}
\usepackage{graphicx}
\usepackage{xcolor}

\usepackage[ansinew]{inputenc}
\usepackage{graphicx}
\usepackage{color}
\usepackage{mathrsfs}
\usepackage{tikz}
\usepackage[colorlinks]{hyperref}
\catcode`\=13
\def{$\bowtie$}

\usepackage{eurosym}
\usepackage{lscape} 

\newtheorem{prethm}{{\bf Theorem}}

\newenvironment{thm}{\begin{prethm}{\hspace{-0.5
				em}{\bf}}}{\end{prethm}}

\newtheorem{prepro}{{\bf Theorem}}

\newtheorem{preprop}{{\bf Proposition}}

\newtheorem{precor}{{\bf Corollary}}

\newenvironment{cor}{\begin{precor}{\hspace{-0.5
				em}{\bf}}}{\end{precor}}

\newtheorem{preconj}{{\bf Conjecture}}

\newtheorem{predefi}{{\bf Definition}}

\newenvironment{defi}{\begin{predefi}{\hspace{-0.5
				em}{\bf}}}{\end{predefi}}

\newtheorem{preremark}{{\bf Remark}}

\newtheorem{preexample}{{\bf Example}}

\newtheorem{prelem}{{\bf Lemma}}

\newenvironment{lem}{\begin{prelem}{\hspace{-0.5
				em}{\bf}}}{\end{prelem}}

\newtheorem{prelam}{{\bf Lemma}}

\newtheorem{preprob}{{\bf Problem}}

\newtheorem{preproof}{{\bf Proof}}

\newenvironment{proof}[1]{\begin{preproof}{\rm
			#1}\hfill{$\Box$}}{\end{preproof}}

\newtheorem{preali}{{\bf Proof of Theorem 1.}}

\newenvironment{ali}[1]{\begin{preali}{\rm
			#1}\hfill{$\Box$}}{\end{preali}}

\newtheorem{prealii}{{\bf Proof of Theorem 2.}}

\newenvironment{alii}[1]{\begin{prealii}{\rm
			#1}\hfill{$\Box$}}{\end{prealii}}

\newtheorem{prealiii}{{\bf Proof of Theorem 3.}}

\newenvironment{aliii}[1]{\begin{prealiii}{\rm
			#1}\hfill{$\Box$}}{\end{prealiii}}

\newtheorem{prealiiii}{{\bf Proof of Theorem 4.}}

\newenvironment{aliiii}[1]{\begin{prealiiii}{\rm
			#1}\hfill{$\Box$}}{\end{prealiiii}}

\newtheorem{prealij}{{\bf Proof of Theorem 5.}}

\newenvironment{alij}[1]{\begin{prealij}{\rm
			#1}\hfill{$\Box$}}{\end{prealij}}

\newtheorem{prealijj}{{\bf Proof of Theorem 6.}}

\newenvironment{alijj}[1]{\begin{prealijj}{\rm
			#1}\hfill{$\Box$}}{\end{prealijj}}

\newtheorem{prealijjj}{{\bf Proof of Theorem 7.}}

\newenvironment{alijjj}[1]{\begin{prealijjj}{\rm
			#1}\hfill{$\Box$}}{\end{prealijjj}}

\newtheorem{prealijjjk}{{\bf Proof of Theorem 8.}}

\author{{\normalsize
		{Arash Ahadi${}^{\mathsf{a}}$},\,
		{Mohsen Mollahajiaghaei${}^{\mathsf{b}}$},\,
		{Ali Dehghan${}^{\mathsf{c}}$},\,
	}\vspace{3mm}
	\\
	{\footnotesize{${}^{\mathsf{a}}$\it Department of Mathematical and Computer Sciences, Kharazmi University, Tehran, Iran}}
	{\footnotesize{}}\\
	{\footnotesize{${}^{\mathsf{b}}$\it Department of Mathematics, University of Western Ontario, London, Ontario, Canada}}
	{\footnotesize{}}\\
	{\footnotesize{${}^{\mathsf{c}}$\it
			Systems and Computer Engineering Department, Carleton University, Ottawa,   Canada}}
	{\footnotesize{}}}

\title{On the maximum number of non attacking rooks
	on a high-dimensional simplicial chessboard}

\begin{document}
\maketitle

\begin{abstract}
{
The simplicial rook graph ${\rm \mathcal{SR}}(m,n)$ is the graph whose vertices are vectors in $ \mathbb{N}^m$ such that for each vector the summation of its coordinates is $n$ and two vertices are adjacent if their corresponding vectors differ in exactly two coordinates. 
Martin and Wagner (Graphs Combin. (2015) 31:1589--1611) asked about the independence number of ${\rm \mathcal{SR}}(m,n)$ that is the  maximum number of  non attacking rooks which can be placed
on a $(m-1)$-dimensional simplicial chessboard of side length $n+1$. In this work, we solve this problem and show that $\alpha({\rm \mathcal{SR}}(m,n))=\big(1-o(1)\big)\frac{\binom{n+m-1}{n}}{m}$. We also prove that  for the domination number of rook graphs we have $\gamma({\rm \mathcal{SR}}(m, n))= \Theta (n^{m-2})$. Moreover we show that these graphs are Hamiltonian. 	
		
The cyclic simplicial rook graph  ${\rm \mathcal{CSR}}(m,n)$  is the graph whose vertices are vectors in $\mathbb{Z}^{m}_{n}$ such that for each vector the summation of its coordinates  modulo $n$ is $0$ and two vertices are adjacent if their corresponding vectors differ in exactly two coordinates. In this work we determine several properties of these graphs such as independence number, chromatic number and automorphism group. Among other results, we also prove that computing the distance between two vertices of a given  ${\rm \mathcal{CSR}}(m,n)$ is $ \mathbf{NP}$-hard in terms of $n$ and $m$. 
}\end{abstract}

\section{Introduction}
\label{sec1}

The simplicial rook graph ${\rm \mathcal{SR}}(m,n)$ is the graph whose vertices are vectors in $ \mathbb{N}^m$ such that for each vector the summation of its coordinates is $n$ and two vertices are adjacent if their corresponding vectors differ in exactly two coordinates.
The simplicial rook graph ${\rm \mathcal{SR}}(3,2)$ is shown in Fig. \ref{fig:sr}. The graph ${\rm \mathcal{SR}}(m,n)$ is  a $n(m-1)$-regular graph and it has $\binom{n+m-1}{n}$ vertices.
These graphs were introduced by Martin and Wagner in 2015 \cite{rook1}. They investigated several properties of  simplicial rook graphs and posed several conjectures and questions about these graphs \cite{rook1}. Martin and Wagner conjectured that ${\rm \mathcal{SR}}(m,n)$ is integral (i.e. all of its eigenvalues are integer numbers) \cite{rook1}. That conjecture was solved by Brouwer {\it et al.} \cite{Rook}.

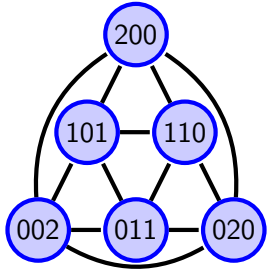
\begin{figure}
	\begin{center}
		\begin{tikzpicture}[baseline=(current bounding box.center),scale=1.3, shorten >=0.9pt, auto, node distance=2cm, ultra thick]
		\begin{scope}[every node/.style={circle,draw=blue,fill=blue!20!,inner sep = 2 pt, ,font=\sffamily}]  
		
		\node (v1) at (4,2) {002};
		\node (v2) at (5,2) {011};
		\node (v3) at (6,2) {020};
		\node (v4) at (4.5,3) {101};
		\node (v5) at (5.5,3) {110};
		\node (v6) at (5,4) {200};    
		
		\draw  (v1) edge  (v4);
		\draw  (v1) edge[bend left]  (v6);
		\draw  (v1) edge  (v2);
		\draw  (v1) edge[bend right]  (v3);
		
		\draw  (v2) edge  (v3);
		\draw  (v2) edge  (v4);
		\draw  (v2) edge  (v5);
		
		\draw  (v3) edge (v5);
		\draw  (v3) edge[bend right]  (v6);
		
		\draw  (v4) edge  (v5);
		\draw  (v4) edge  (v6);
		
		\draw  (v5) edge  (v6);
		\end{scope}
		\end{tikzpicture}
	\end{center}
	\caption{ The simplicial rook graph $ {\rm \mathcal{SR}} (3, 2)$.}\label{fig:sr}
\end{figure}

Martin and Wagner also asked about the independence number of ${\rm \mathcal{SR}}(m,n)$ that can be interpreted as   the  maximum
number of  non attacking rooks which can be placed
on a $(m-1)$-dimensional simplicial chessboard of side length $n+1$ \cite{rook1}. Finding the maximum
number of  non attacking rooks which can be placed
on a $2$-dimensional simplicial chessboard (i.e. calculating  the independence number of ${\rm \mathcal{SR}}(3,n)$) was independently investigated by several authors. Nivasch and Lev \cite{Nivasch}, Blackburn {\it et al.} \cite{Blackburn} and Vaderlind {\it et al.} \cite{vaderlind} independently proved that  the independence number of ${\rm \mathcal{SR}}(3,n)$ is equal to $1+ \lfloor   \frac{2}{3} n  \rfloor $.
Recently, Brouwer {\it et al.}  calculated  the independence number of ${\rm \mathcal{SR}}(m,3)$  \cite{Rook}. In this work
we present tight lower and upper bounds for the independence number of  ${\rm \mathcal{SR}}(m,n)$ for any $m$ and $n$. For more details about the connection between 
the independence number of ${\rm \mathcal{SR}}(m,n)$ and the  maximum
number of  non attacking rooks please see \cite{rook1}.

\begin{thm}\label{Th01}
Let $p$ be a prime number such that $p\geq  \max\{m,n\}$, then $\frac{\binom{n+m-1}{n}}{p} \leq \alpha({\rm \mathcal{SR}}(m,n)) \leq \frac{\binom{n+m-1}{n}}{m}$. Also, 
  $\alpha({\rm \mathcal{SR}}(m,n)) =\big(1-o(1)\big)\frac{\binom{n+m-1}{n}}{m}$.
\end{thm}

Motivated by the problem of determining the independence number,  we also study its domination number.
Next, we focus on domination number of $SR(m,n)$ and find the asymptotic value of it.

\begin{thm}\label{Th02}
In terms of $n$, we have $\gamma({\rm \mathcal{SR}}(m,n))= \Theta (n^{m-2})$.	
\end{thm}

The graph properties such as diameter, chromatic number and Hamiltonicity are all related to the spectrum.  Vermette in has PhD thesis studied the spectra of the  simplicial  rook  graphs and how their spectra relate to their properties \cite{Jason}. Finally, we study the remaining properties of these graphs.
We show that the simplicial rook graph  $ {\rm \mathcal{SR}} (m ,n)$ is Hamiltonian for any values of $m$ and $n$, except  the  cases $m=1$ and $(m,n)=(2,1)$.

\begin{thm}\label{Th03}
The simplicial rook graph  $ {\rm \mathcal{SR}} (m ,n)$ is Hamiltonian for any values of $m$ and $n$, except  the  cases $m=1$ and $(m,n)=(2,1)$.	
\end{thm}

In the second part of the work, we introduce and investigate a cyclic version of  simplicial rook graphs. The cyclic simplicial rook graph  ${\rm \mathcal{CSR}} (m,n)$  is the graph whose vertices are vectors in $\mathbb{Z}^{m}_{n}$ such that for each vector the summation of its coordinates  modulo $n$ is $0$ and two vertices are adjacent if their corresponding vectors differ in exactly two coordinates. The cyclic simplicial rook graph ${\rm \mathcal{CSR}}(4,2)$ is shown in Fig. \ref{fig:csr}. The graph $ {\rm \mathcal{CSR}} (m,n)$ is a $\binom{m}{2}(n-1)$-regular graph with $n^{m-1}$ vertices. In this work we study several properties of these graphs such as their chromatic, independence  and clique numbers. 

\begin{thm}\label{Th04}
(i)	Let $p$ be a prime number such that $p\geq  \max\{m, n\}$. Then $m \leq \chi\big( {\rm \mathcal{CSR}}(m,n)\big) \leq p$.\\
(ii) Let $p$ be a prime number such that $p\geq  \max\{m,n\}$, then $\frac{\binom{n+m-1}{n}}{p} \leq \alpha({\rm \mathcal{CSR}}(m,n)) \leq \frac{\binom{n+m-1}{n}}{m}$.\\
(iii) If $m, n \neq 1$, then $\omega( {\rm \mathcal{CSR}} (m,n))= \max \{n,m\}$.
\end{thm}

In this work we also study the diameter of the  cyclic simplicial rook graphs.

\begin{thm}\label{Th05}
We have	$diam( {\rm \mathcal{CSR}} (m,n))=m-\lfloor \frac{m-1}{n} \rfloor -1$.
\end{thm}

The distance between two vertices is equal to the number of edges of a shortest path between those vertices.
We show that computing the distance between the vertices is an NP-hard problem in term of $m$ and $n$.

\begin{thm}\label{Th06}
Computing the distance between two vertices of a given  ${\rm \mathcal{CSR}}(m,n)$ is $ \mathbf{NP}$-hard in terms of $n$ and $m$.
\end{thm}

In \cite{Rook} it is shown that the automorphism group of ${\rm \mathcal{SR}}(m,n)$ is $\mathcal{S}_m$ when $n>3$, where $\mathcal{S}_m$ stands for the permutation group on $m$ elements.  Also it is $\mathcal{S}_m \times \mathbb{Z}_2$ when $n=3$.
Finally, we study the automorphism group of ${\rm \mathcal{CSR}}(m,n)$.

\begin{thm}\label{Th07}
We have $\mathbf{Aut}( {\rm \mathcal{CSR}}(m,n))\cong \mathcal{S}_m \times  \
	\mathbb{Z}^{\times}_{n} \times \mathbb{Z}^{m-1}_{n}$ for $n, m>3$.
\end{thm}

\begin{figure}
	\begin{center}
		\begin{tikzpicture}[baseline=(current bounding box.center),scale=1.2, 
		shorten >=0.7pt, auto, node distance=2cm, ultra thick]
		\begin{scope}[every node/.style={circle,draw=blue,fill=blue!20!,inner sep = 2 pt, ,font=\sffamily}]  
		
		\node (v1) at (4,2) {0000};
		\node (v2) at (5,2) {1111};
		\node (v3) at (6,1) {1100};
		\node (v4) at (6,0) {0011};
		\node (v5) at (5,-1) {1010};
		\node (v6) at (4,-1) {0101};    
		\node (v7) at (3,0) {1001};    
		\node (v8) at (3,1) {0110};    
		
		\draw  (v1) edge  (v3);
		\draw  (v1) edge  (v4);
		\draw  (v1) edge  (v5);
		\draw  (v1) edge  (v6);
		\draw  (v1) edge  (v7);
		\draw  (v1) edge  (v8);
		
		\draw  (v2) edge  (v3);
		\draw  (v2) edge  (v4);
		\draw  (v2) edge  (v5);
		\draw  (v2) edge  (v6);
		\draw  (v2) edge  (v7);
		\draw  (v2) edge  (v8);
		
		\draw  (v3) edge  (v5);
		\draw  (v3) edge  (v6);
		\draw  (v3) edge  (v7);
		\draw  (v3) edge  (v8);
		
		\draw  (v4) edge  (v5);
		\draw  (v4) edge  (v6);
		\draw  (v4) edge  (v7);
		\draw  (v4) edge  (v8);
		
		\draw  (v5) edge  (v7);
		\draw  (v5) edge  (v8);
		
		\draw  (v6) edge  (v7);
		\draw  (v6) edge  (v8);
		
		\end{scope}
		\end{tikzpicture}
	\end{center}
	\caption{The cyclic simplicial rook graph ${\rm \mathcal{CSR}} (4, 2)$. }\label{fig:csr}
\end{figure}
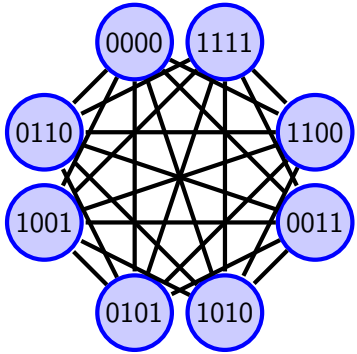

 Terminology and notation generally follow \cite{MR1367739}. 
Throughout the work, we denote by $[n]$ the set of integers $\{1, \dots , n\}$, also the graphs have no parallel edges and no loops. Furthermore, for a given graph $G$, we denote its vertex set by $V(G)$ and its edge set by  $E(G)$. Note that for an element $x$ and a set $A$, we have
$x+A=\{   x+a | a \in A  \}$.

\section{The simplicial rook graph ${\rm \mathcal{SR}}(m,n)$}\label{sec3}

\subsection{Independence number of  ${\rm \mathcal{SR}}(m,n)$}

Here, we prove that if $p$ is a prime number such that $p\geq  \max\{m,n\}$, then $\frac{\binom{n+m-1}{n}}{p} \leq \alpha({\rm \mathcal{SR}}(m,n)) \leq \frac{\binom{n+m-1}{n}}{m}$. Also, 
$\alpha({\rm \mathcal{SR}}(m,n)) =\big(1-o(1)\big)\frac{\binom{n+m-1}{n}}{m}$.

\begin{ali}{
First we prove the lower bound. For each nonnegative integer $k$, define $\mathcal{P}_{k}=\{ (a_{1},\ldots,a_{m})\in  {\rm \mathcal{ SR}} (m,n)\mid \sum_{i=1}^{m}ia_{i}=k \}$. The sets $\{\mathcal{P}_{k}\, |\, k \in \mathbb{N}\}-\{\emptyset\}$ is a partitioning for the vertices of ${\rm \mathcal{SR}} (m,n)$.
Let $p\geq \max\{m,n\}$ be a prime number. 
Define $\Upsilon_{t}=\bigcup_{k\equiv t~mod~p}\mathcal{P}_{k}$.
Then $\bigcup_{t\in [p]}\Upsilon_{t}$ is a partitioning for the vertices of ${\rm \mathcal{SR}}(m,n)$.

Next, we show that for each $t$, $\Upsilon_{t}$ is an independent set.
To the contrary suppose that $t\in [p]$ and $a,b\in \Upsilon_{t}$ are adjacent vertices in ${\rm \mathcal{SR}}(m,n)$. 
Let $a=(a_{1},\ldots,a_{m})$ and $b=(b_{1},\ldots,b_{m})$. 
Thus,  there are integers $i$ and $j$ such that $i\neq j$, and 
\begin{align} 
&a_{i}\neq b_{i},~ ~a_{j}\neq b_{j},& \text{Since $a$ and $b$ are adjacent,} \\
&a_{i}+a_{j}=b_{i}+b_{j}, &\text{Since $a ,b \in V({\rm \mathcal{SR}}(m,n))$,} \\
&ia_{i}+ja_{j}\equiv ib_{i}+jb_{j}~ (mod~p),  &\text{Since $a ,b \in \Upsilon_{t}$.}
\end{align}
By using (2) and (3), we have $(i-j)(a_{i}-b_{i})\equiv 0~ (mod~p)$. By (1), and noting that $p>\max \{m, n\}$ and it is a prime number, we  lead to a contradiction.  Thus,  for each $t$, $\Upsilon_{t}$ is an independent set. By the pigeonhole principle there is an integer $i$ such that the independent set $\Upsilon_{i}$ has at least $\frac{ V({\rm \mathcal{SR}}(m,n))}{p}=\frac{\binom{n+m-1}{n}}{p}$ vertices. This proves the lower bound.

Next, we focus on upper bound. We use the Hoffman's bound for independence number which states that if $G$ is   an $r$-regular graph on $V(G)$ vertices whose adjacency matrix $A$ has least eigenvalue $\lambda_{min}$, then $\alpha(G) \leq \frac{-\lambda_{min}}{r-\lambda_{min}} |V(G)|$ \cite{hof}.  
For the graph ${\rm \mathcal{SR}}(m,n)$ we know that $\lambda_{\min}=max \{ -n, -\binom{m}{2} \}$ \cite{Rook}. Also, note that  the graph ${\rm \mathcal{SR}}(m,n)$ is  a $n(m-1)$-regular graph and it has $\binom{n+m-1}{n}$ vertices \cite{rook1}. Thus, by the Hoffman's bound we have two cases:   If  $n \leq \binom{m}{2}$, then
\begin{equation}\label{E001}
\alpha \leq \frac{-\lambda_{min}}{r-\lambda_{min}}\binom{n+m-1}{n} = \frac{n}{n(m-1)+n}\binom{n+m-1}{n}= \frac{\binom{n+m-1}{n}}{m}.
\end{equation}
Also, if   $n> \binom{m}{2}$, then
\begin{equation}\label{E002}
\alpha \leq \frac{-\lambda_{min}}{r-\lambda_{min}}\binom{n+m-1}{n} = \frac{\binom{m}{2} \binom{n+m-1}{n}}{n(m-1)+\binom{m}{2}}  \leq  \frac{\binom{m}{2} \binom{n+m-1}{n}}{\binom{m}{2}(m-1)+\binom{m}{2}} = \frac{\binom{n+m-1}{n}}{m}.
\end{equation}
This completes the proof of upper bound.

Finally, we show that $\alpha({\rm \mathcal{SR}}(m,n)) =\big(1-o(1)\big)\frac{\binom{n+m-1}{n}}{m}$. 
It follows from the prime number theorem that for any real $\epsilon>0$ there is a $M> 0 $ such that for all $M' > M$ there is a prime $p$ such that $M' < p < (1 + \epsilon) M'$ (see \cite{PNT},  page 494).
Having the lower and upper bounds and using the prime number theorem  lead to $\alpha({\rm \mathcal{SR}}(m,n)) =\big(1-o(1)\big)\frac{\binom{n+m-1}{n}}{m}$. 		
}\end{ali}

\subsection{Domination number of  ${\rm \mathcal{SR}}(m,n)$}

Here, we prove that in terms of $n$ ($m$ is a fixed number), we have $\gamma({\rm \mathcal{SR}}(m,n))= \Theta (n^{m-2})$.

\begin{alii}{
It is well-known that for every graph $G$ we have $\gamma(G) \geq \frac{|V(G)|}{\Delta(G)+1}$ \cite{MR1367739}. Thus, 
\begin{equation}\label{E000}
 \gamma({\rm \mathcal{SR}}(m, n)) \geq \frac{\binom{n+m-1}{m-1}}{n(m-1)+1}= \Omega(n^{m-2}).
\end{equation} 

Next, consider the  set $D= \{(x_1, ..., x_m) \,|\, x_1=x_2\}$). We show that $D$ is a dominating set for
${\rm \mathcal{SR}}(m, n)$. Let $a=(a_{1},\ldots,a_{m})$ be an arbitrary vertex of ${\rm \mathcal{SR}}(m, n)$. Three cases can be considered: 
(i) If $a_2 > a_1$, then the vertex $a$ is adjacent with $(a_{1},a_{1},a_{3}-a_{1}+a_{2},a_{4},\ldots,a_{m})\in D$. (ii) Also, if $a_1>a_2$ then the vertex $a$ is adjacent with  $(a_{2},a_{2},a_{3}-a_{2}+a_{1},a_{4},\ldots,a_{m})\in D$. (iii) If $a_1=a_2$, then $a\in D$.
Therefore, $D$ is a dominating set. Next, we calculate the number of vertices in $D$. 
We note that $|D|$ is equal to $\sum_{i=0}^{\lfloor\frac{n}{2}\rfloor}|V(\mathcal{SR}(m-2,n-2i))|$. Consequently, 
\begin{equation}\label{E003}
|D| = \sum_{i=0}^ {\lfloor n/2 \rfloor } \binom{n+m-3-2i}{m-3}. 
\end{equation}
Note that for every $i\in \{0,1,\ldots,\lfloor n/2 \rfloor \}$, we have 
\begin{equation}\label{E004}
\binom{n+m-3-2i}{m-3} \leq \Big( \binom{n+m-3-2i}{m-3}+ \binom{n+m-3-2i+1}{m-3} \Big)/2
\end{equation}
By substituting (\ref{E004}) in (\ref{E003}), we have 
\begin{equation}\label{E005}
|D| \leq    \frac{1}{2} \sum_{i=-1}^ {2 \lfloor   n/2 \rfloor} 
\binom{n+m-3-i}{m-3} 
\end{equation}
By use Pascal's rule which states $\binom{x-1 }{y } + \binom{x-1 }{y-1 }=\binom{x }{y } $ and (\ref{E005}), we have
\begin{equation}\label{E006}
|D| \leq   \frac{1}{2} \binom{n+m-1}{m-2}= \mathcal{O}(n^{m-2}).
\end{equation}
By combining (\ref{E000}) and (\ref{E006}), we have $\gamma({\rm \mathcal{SR}}(m,n))= \Theta (n^{m-2})$.
}
\end{alii}

We conjecture that $\{(i,i,n-2i)\mid 0 \leq i \leq \lfloor n/2 \rfloor  \}$ is a minimum dominating set for ${\rm \mathcal{SR}}(3, n)$.

\subsection{Hamiltonian graphs} 

Here, we show that the simplicial rook graph  $ {\rm \mathcal{SR}} (m ,n)$ is Hamiltonian for any values of $m$ and $n$, except  the  cases $m=1$ and $(m,n)=(2,1)$.	

\begin{aliii}{
We prove an statement that is stronger than the theorem's statement. We show that every simplicial rook graph  $ {\rm \mathcal{SR}} (m ,n)$,  except  the  cases $m=1$ and $(m,n)=(2,1)$, has a Hamiltonian cycle $C$ such that it contains the edge  $e_n^m= (n, 0^{m-1}) (n-1, 1, 0^{m-2})$. Let $m\geq 3$ be an arbitrary number, then the graph  $ {\rm \mathcal{SR}} (m ,1)$ is a complete graph $K_{m}$, so it has a Hamiltonian cycle that go through the edge  $(1, 0^{m-1}) (0, 1, 0^{m-2})$.  We prove the remaining cases by induction on $m$. 
\\
\underline{Base case (i.e. $m=2$)}:  Consider the  simplicial rook graph   $ {\rm \mathcal{SR}} (2 ,n)$, where $n\geq 2$. The graph  $ {\rm \mathcal{SR}} (2 ,n)$ is  a complete graph $K_{n+1}$, thus, it has a Hamiltonian cycle that go through the edge  $(n, 0^{m-1}) (n-1, 1, 0^{m-2})$. \\
\underline{Inductive step (i.e. $m\geq 3$)}: Consider the graph  $ {\rm \mathcal{SR}} (m ,n)$, where $m\geq3$ and $n\geq 2$. We show that it has a Hamiltonian cycle $C$ such that it contains the edge  $e_n^m= (n, 0^{m-1}) (n-1, 1, 0^{m-2})$. By the induction hypothesis, for each $k \in [n]$, the graph  $\mathcal{SR}(m -1, k)$ 
has a Hamiltonian cycle $\mathcal{C}_k$ that contains the edge $e_k^{m-1}= (k, 0^{m-2}) (k-1, 1, 0^{m-3})$. Let $H_k^{m-1}$ be the set of edges $E(\mathcal{C}_k)\setminus \{e_k^{m-1}\}$. The following edges form a Hamiltonian cycle $C$ for $\mathcal{SR}(m, n)$.
\\ \\
$\bigcup^n _{k=1}  \big\{(n-k , u) (n-k, v) | uv \in H_k^{m-1} \big\}$,
\\
$\bigcup^{n} _{k=2} \big\{(n-k, k-1, 1, 0^{m-3}) (n-k+1, k-1, 0^{m-2}) \big\}$,
\\ 
$\bigcup \big\{(n-1, 0, 1, 0^{m-3}) (n, 0^{m-1}) \big\}$
$\bigcup \big\{(n, 0^{m-1}) (0, n, 0^{m-2}) \big\}$,
\\
where $(n-k , u)$ is a vector such that its first coordinate is $n-k$ and its other coordinates are the coordinates of the vector $u$.

Note that $\bigcup^n _{k=1}  \big\{(n-k , u) (n-k, v) | uv \in H_k^{m-1} \big\}$ are $n$ paths $\mathcal{P}_1,\mathcal{P}_2,\ldots,\mathcal{P}_n $ forming a partition of the vertices of $\mathcal{SR}(m,n)$. On the other hand, $\bigcup^{n} _{k=2} \big\{(n-k, k-1, 1, 0^{m-3}) (n-k+1, k-1, 0^{m-2}) \big\}$ are edges to join the last vertex of the path $\mathcal{P}_i$ to the path  $\mathcal{P}_{i+1}$.

Without loss of generality assume that $C=v_1,v_2,\ldots,v_z$.
Let $f:V(\mathcal{SR}(m, n)) \rightarrow V(\mathcal{SR}(m, n)) $ be a function that assigns a given vertex $v_i=(a_1,a_2,a_3,a_4\ldots,a_m)$ to the vertex $v_i'=(a_1,a_3,a_2,a_4\ldots,a_m)$. In fact it switches the second and third coordinates. Noting that $C$ is a Hamiltonian cycle containing the edge  $(n, 0^{m-1}) (n-1, 0, 1, 0^{m-3})$, thus, $C'=f(v_1),f(v_2),\ldots,f(v_z)$ is a Hamiltonian cycle containing the edge  $(n, 0^{m-1}) (n-1, 1, 0, 0^{m-3})$. This completes the proof.
}
\end{aliii}

\section{The cyclic simplicial rook graph  ${\rm \mathcal{CSR}} (m,n)$}\label{sec4}

\subsection{Chromatic and clique numbers of ${\rm \mathcal{CSR}}(m,n)$}

Here, we prove the following:\\
(i)	Let $p$ be a prime number such that $p\geq  \max\{m, n\}$. Then $m \leq \chi\big( {\rm \mathcal{CSR}}(m,n)\big) \leq p$.\\
(ii) Let $p$ be a prime number such that $p\geq  \max\{m,n\}$, then $\frac{\binom{n+m-1}{n}}{p} \leq \alpha({\rm \mathcal{CSR}}(m,n)) \leq \frac{\binom{n+m-1}{n}}{m}$.\\
(iii) If $m, n \neq 1$, then $\omega( {\rm \mathcal{CSR}} (m,n))= \max \{n,m\}$.

\begin{aliiii}
{
(i) We first prove the lower bound. Let $f:V({\rm \mathcal{SR}}(m,n-1))\longrightarrow V({\rm \mathcal{CSR}}(m,n))$ be a function such that $f(x_{1},\ldots,x_{m})=(x_{1},\ldots,x_{m}+1)$. The function $f$ is a graph homomorphism, thus, 
\begin{equation}\label{E007}
\chi({\rm \mathcal{SR}}(m,n-1))\leq \chi ({\rm \mathcal{CSR}}(m,n)).
\end{equation}
We use the Hoffman's bound  for chromatic number which states that if $G$ is   an $r$-regular graph on $V(G)$ vertices whose adjacency matrix $A$ has least eigenvalue $\lambda_{min}$, then $\chi(G) \geq \frac{r-\lambda_{min}}{-\lambda_{min}}  $ \cite{hof}.  
For the graph ${\rm \mathcal{SR}}(m,n)$ we know that $\lambda_{\min}=max \{ -n, -\binom{m}{2} \}$ \cite{Rook}. Also, note that  the graph ${\rm \mathcal{SR}}(m,n)$ is  a $n(m-1)$-regular graph \cite{rook1}. We have two cases:   If  $n \leq \binom{m}{2}$, then by  the Hoffman's bound  and (\ref{E007}) we have 
\begin{equation}\label{E008}
\chi ({\rm \mathcal{CSR}}(m,n)) \geq  \chi({\rm \mathcal{SR}}(m,n-1)) \geq \frac{r-\lambda_{min}}{-\lambda_{min}} \geq  \frac{n(m-1)+n}{n}=m.
\end{equation}
Also if $n> \binom{m}{2}$, we have
\begin{equation}\label{E009}
\chi ({\rm \mathcal{CSR}}(m,n)) \geq  \chi({\rm \mathcal{SR}}(m,n-1)) \geq \frac{r-\lambda_{min}}{-\lambda_{min}} \geq \frac{n(m-1)+\binom{m}{2}}{\binom{m}{2}} > \frac{\binom{m}{2}(m-1)+\binom{m}{2}}{\binom{m}{2}}  =m.
\end{equation}
For the upper bound, we use the proof of Theorem \ref{Th01}.
In the proof of that theorem we showed that  $\bigcup_{t\in [p]}\Upsilon_{t}$ is a partitioning for the vertices of ${\rm \mathcal{SR}}(m,n)$, where $p\geq  \{m, n\}$ and  for each $t$, the set $\Upsilon_{t}$ is an independent set. The sets  $\bigcup_{t\in [p]}\Upsilon_{t}$  is also a partitioning for the vertices of ${\rm \mathcal{CSR}}(m,n)$. For each set $\Upsilon_{t}$ by coloring its vertices by the color $t$, we obtain a proper coloring with $p$ colors. This proves the upper bound.\\
(ii) The proof of this part is completely similar to the proof of 	Theorem \ref{Th01}.\\
(iii) The graph ${\rm \mathcal{CSR}}(m,n)$ is vertex-transitive\footnote{A graph $G$ is vertex-transitive if for any two vertices $v_1$ and $v_2$ of $G$ there is some automorphism $f:V(G)\rightarrow V(G)$ such that $f(v_1)=v_2$}. So we can assume that  $W$ is a maximum clique and $\mathbf{0}=0^m\in W$. 
Each neighbor of the vertex $\mathbf{0}$ has the form $\alpha (e_{a}-e_{b})$ for some $\alpha \neq 0$ and $a \neq b$, where $e_{i}$ is a vector whose coordinates are all zero, except the $i$th coordinate that is equal to one. Suppose that $\mathbf{s}=\alpha(e_{a}-e_{b})\in W$.
The common neighbors of the vertex $\mathbf{0}$ and the vertex  $\mathbf{s}$ are the union of sets 
$N_1:=\{\beta(e_{a}-e_{b})\mid \beta\in \mathbb{Z}_{n}-\{0,\alpha\}\}$, 
$N_2:=\{\alpha(e_{t}-e_{b})\mid t\in
[m]-\{a,b\} \}$ and $N_{3}:=\{ \alpha(e_{a}-e_{r})\mid r\in [m]-\{ a,b \} \}$. 
We have $|N_1|=n-2$ and $|N_{2}|=|N_{3}|=m-2$. Note that $N_{1}\cap N_{2}=N_{1}\cap N_{3}=N_{2}\cap N_{3}=\emptyset$. Also, there is no edge between $N_{1}$ and $N_{2}$ and no edge  between $N_{1}$ and $N_{3}$ and if  $\alpha\neq \frac{n}{2}$, then there is no
edge between $N_{2}$ and $N_{3}$. If $\alpha= \frac{n}{2}$ (i.e., $\alpha=-\alpha$) 
then for every $t \in [m]-\{a, b\}$, 
the vertex $\alpha(e_t+e_b)$ of $N_2$ has a unique neighbor in $N_3$ 
(that is the vertex $\alpha(e_a+e_t)$) and vise versa. 
Thus, the induced subgraph between $N_{2}$ and $N_{3}$ is a perfect matching. 
Consequently, $W \cap (N_{1}\cup N_{2}\cup N_{3})=  \max
\{n-2,m-2\}$. Thus, $\omega({\rm \mathcal{CSR}}(m,n))= \max
\{n,m\}$, which proves the theorem.
}
\end{aliiii}

By the proof of part (iii) of Theorem \ref{Th04}, we have the following corollary.

\begin{cor}\label{cliqueunique}
Every maximum clique $W$ has the following structure:
\begin{itemize}
	\item If $n>m$, then $W$ is a coset of $S_{ab}$, where $S_{ab}=\{ \beta(e_{a}-e_{b})\mid \beta\in \mathbb{Z}_{n}\}$;
	in other words, there are $x\in \mathbb{Z}_{n}^{m}$ and $a \neq b$ such that $W=x+S_{ab}$.
	\item If $n<m$, then there are $\beta\in \mathbb{Z}_{n}$  and $\xi \in \mathbb{Z}_{n}^{m}$ 
	such that $W=\xi + \beta\{e_a | a \in [m]\}$.
	\item If $n=m$, then any of above structures can be happened.
\end{itemize}
\end{cor}

\subsection{Diameter of ${\rm \mathcal{CSR}}(m,n)$}

In \cite{Rook} it is shown that $diam\big(\mathcal{SR}(m,n)\big)= \min \{m-1, n\}$.
Here we compute the diameter of ${\rm \mathcal{CSR}}(m,n)$. Before we present the proof, we need to define some notations.

\begin{defi}
Let $\mathbf{b}=(b_{1},\ldots,b_{m})$ be a vertex of ${\rm \mathcal{CSR}}(m,n)$. A zero partitioning of size $t$  for the vertex $\mathbf{b}$ is a partitioning of the set $[m]$ into $t$ parts $P_{1},\ldots,P_{t}$ such that $\sum_{j\in P_{i}}b_{j}=0$ modulo $n$ for every $i \in [t]$. The zero partitioning number of a vertex $\mathbf{b}$, denoted by $\tau(\mathbf{b})$, is the maximum $t$ that there exists a zero partitioning of size $t$ for the vertex $\mathbf{b}$.
\end{defi}

For example, $\{ \{1,2,6\},\{4\},\{3,5\}
\}$ is a zero partitioning of size three for the vertex $\mathbf{b}=(2,2,1,0,2,2) \in \mathbb{Z}^6_{3}$ and it is easy to check that  $\tau(\mathbf{b})=3$.

The distance between two vertices $v$ and $u$, denoted by $d(v,u)$ is equal to the number of edges of a shortest path between those vertices.

\begin{lem}\label{distance}
Let $\mathbf{0}=0^m$ and $\mathbf{b}$ be two vertices  of ${\rm \mathcal{CSR}}(m,n)$. Then $d(\mathbf{0},\mathbf{b})=m-\tau(\mathbf{b})$.
\end{lem}

\begin{proof} 
{Let $\mathbf{b}=(b_1,b_2,\ldots,b_m)$ and $\mathbf{0}=0^m$. First we show that $d(\mathbf{0},\mathbf{b})\leq m-\tau(\mathbf{b})$.  Let $P_{1},\ldots,P_{\tau(\mathbf{b})}$ be a zero partitioning of size $\tau(\mathbf{b})$  for the vertex $\mathbf{b}$. For each $t\in [\tau(\mathbf{b})]$, let $v_t$ be the vertex in  ${\rm \mathcal{CSR}}(m,n)$ such that all coordinates of $v_t$ that are in $\cup_{i=t+1}^{\tau(\mathbf{b})} P_i$ are zero and all of its other coordinates are equal to the corresponding coordinates in $\mathbf{b}$. Also, put $v_0=\mathbf{0}=0^m$.
For each $t$, there is a path of length at most $|P_{t}|-1$ between the vertices $v_{t-1}$ and $v_t$. So, there is a path of length at most $\sum_{i=1}^{\tau(\mathbf{b})}|P_{i}|-1=m-\tau(\mathbf{b})$ between $\mathbf{b}$ and $\mathbf{0}$.

Next, by induction on $m$, we prove that $d(\mathbf{0},\mathbf{b})\geq  m-\tau(\mathbf{b})$. The claim is clearly true for $m=1,2$. By induction hypothesis for every $m'<m$ and $\mathbf{b'}=(b_1',b_2',\ldots,b_{m'}')\in {\rm \mathcal{CSR}}(m',n)$, we have $d(\mathbf{0},\mathbf{b'})\geq  m'-\tau(\mathbf{b'})$. Now, we show that $d(\mathbf{0},\mathbf{b})\geq  m-\tau(\mathbf{b})$. To the contrary, suppose that
$d(\mathbf{0},\mathbf{b})< m-\tau(\mathbf{b})$. This   means that we have $m-\tau(\mathbf{b})-1$ vectors
like $x_{j}=a_{j}(e_{i_{j}}-e_{k_{j}})$
such that $\sum_{j=1}^{m-\tau(\mathbf{b})-1}x_{j}=\mathbf{b}$. Note that for each  $x_{j}=a_{j}(e_{i_{j}}-e_{k_{j}})$, we have $a_{j}\in \mathbb{Z}_{n}$, and also $e_{i_{j}}$ and $e_{k_{j}}$ are   vectors such that their coordinates are all zero, except the $i_{j}$th
coordinate (respectively,  $k_{j}$th coordinate) that is equal to one.
Let $\alpha_q=|\big\{j\mid e_q\in \{e_{i_{j}},e_{k_{j}}\}\big\}|$. So
$\sum_{q=1}^{m}\alpha_q=2(m-\tau(\mathbf{b})-1)$.
It means that there exists $r>1$ such that $\alpha_{r}=1$ 
(if $b_{i}\neq 0$ then $\alpha_{i}> 0$). 
Let $\{k\} =\big\{j\mid e_{r}\in\{e_{i_{j}},e_{k_{j}}\}\big\}$. 
Since the order of $x_{i}$ does not
matter, we can assume that $k=m-\tau(\mathbf{b})-1$. 
So $a_{m-\tau(\mathbf{b})-1}=b_{r}$.
Let  $\mathbf{c}= \sum _{i=1}^{m-\tau(\mathbf{b})-2} x_i$. 
The $r$-th coordinate of $c$ and $x_{i}$ for
$i=1,\ldots,m-\tau(\mathbf{b})-2$ are zero.  
Remove $r$-th coordinate in $x_{i}$ 
and $\mathbf{c}$ 
and call them $x'_{i}$ and $c'$, respectively. Note that $\mathbf{c'}$ is a vertex in ${\rm \mathcal{CSR}}(m-1,n)$. By the structure of $\mathbf{c'}$ we have
\begin{equation}\label{E010}
\tau(\mathbf{b})  \geq \tau(c') .
\end{equation}  
Also, by the induction hypothesis, we have
\begin{equation}\label{E011}
d_{\mathcal{CSR}(m-1,n)}(\mathbf{0},c')=m-1-\tau(c').
\end{equation}
By substituting (\ref{E010}) in (\ref{E011}), we have
\begin{equation}\label{E013}
d_{{\rm \mathcal{CSR}}(m-1,n)}(\mathbf{0},c')\geq m-1-\tau(\mathbf{b}).
\end{equation}	
On the other hand, the sequence $\mathbf{0},x'_{1},x'_{1}+x'_{2},\ldots,x'_{1}+\cdots
+x'_{m-\tau(\mathbf{b})-2}=c'$ 
is a walk of length $m-\tau(\mathbf{b})-2 $ from the vertex $\mathbf{0}$ to the vertex $c'$ in ${\rm \mathcal{CSR}}(m-1,n)$. 
Thus,
\begin{equation}\label{E012}
d_{{\rm \mathcal{CSR}}(m-1,n)}(\mathbf{0},c')\leq m-\tau(\mathbf{b})-2.
\end{equation}
The equations (\ref{E013}) and (\ref{E012}) together form a contradiction. This completes the proof of induction and thus completes the proof of lemma.
}\end{proof}

Before we prove the main result we need to present one more lemma.
Let $\Gamma$ be a group and $S$ be a subset of group elements  such that $S$ is closed under inverses and also $e \notin S$, where $e$ is the identity element  of $\Gamma$.
Then, the Cayley graph $Cay(\Gamma, S)$ is a graph where its vertices 
are the members of $\Gamma$ and two members $x$ and $y$ are adjacent if and only if $x^{-1}y \in S$.
Now consider the subgroup $\Gamma=\{(a_{1},\ldots,a_{m})\in \mathbb{Z}_{n}^{m} \mid \sum_{i=1}^{m}a_{i}=0\}$ of $\mathbb{Z}_{n}^{m}$
and let $S$ be the subset of $\Gamma$ 
such that for each member of $S$, all coordinates except two of them are zero. 
Then,  $\mathcal{CSR}(m,n)\cong Cay(\Gamma,S) $. 

 \begin{lem}\label{L001}
The cyclic simplicial rook graph  ${\rm \mathcal{CSR}} (m,n)$  is a Cayley graph.
 \end{lem}

Now we are ready to calculate the diameter of ${\rm \mathcal{CSR}}(m,n)$.

\begin{alij}{
Consider the cyclic simplicial rook graph  ${\rm \mathcal{CSR}} (m,n)$. By Lemma 
\ref{L001},  ${\rm \mathcal{CSR}}(m,n)$ is a Cayley graph, so we have	
\begin{equation}\label{discay}
diam({\rm \mathcal{CSR}}(m,n))=\max\{d(\mathbf{a},\mathbf{b})\mid \mathbf{a},\mathbf{b}\in
{\rm \mathcal{CSR}}(m,n)\}=\max\{d(\mathbf{0},\mathbf{b})\mid \mathbf{b}\in
{\rm \mathcal{CSR}}(m,n)\}.
\end{equation}
By using Lemma \ref{distance} in (\ref{discay}), we have
\begin{equation}\label{E014}
 diam({\rm \mathcal{CSR}}(m,n))=\max\{m-\tau(\mathbf{b})\mid \mathbf{b}
 \in {\rm \mathcal{CSR}}(m,n)\}.
\end{equation}
By (\ref{E014}) to prove the theorem it is enough to show that 
\begin{equation}\label{E015}
\min\{\tau(\mathbf{b})\mid \mathbf{b} \in
{\rm \mathcal{CSR}}(m,n)\}=1+ \lfloor  \frac{m-1}{n} \rfloor.
\end{equation}
Let
$\mathbf{b}=(n-(m-1),1^{m-1})$. Then, we have
$\tau(\mathbf{b})=1+\lfloor  \frac{m-1}{n} \rfloor$. So,
\begin{equation}\label{E016}
\min\{\tau(\mathbf{b})\mid \mathbf{b}
\in {\rm \mathcal{CSR}}(m,n)\}\leq 1+ \lfloor  \frac{m-1}{n}  \rfloor.
\end{equation}
Next, we show that for each vertex $\mathbf{b}$ in ${\rm \mathcal{CSR}}(m,n)$, we have
\begin{equation}\label{E017}
\tau(\mathbf{b} )\geq  1+ \lfloor  \frac{m-1}{n} \rfloor
\end{equation}
For  given $m$ and $n$, in order to prove (\ref{E017}), we consider two cases:\\
If $m\leq n$, then $\min\{\tau(\mathbf{b})\mid \mathbf{b}\in {\rm \mathcal{CSR}}(m,n)\}\leq  1+\lfloor  \frac{m-1}{n}  \rfloor=1$. So, $\min\{\tau(\mathbf{b})\mid \mathbf{b}\in {\rm \mathcal{CSR}}(m,n)\}=1$.\\
If $n\leq m$, we prove (\ref{E017}) by induction on
$m$. Let $\mathbf{b}=(b_{1},\ldots,b_{m})$ be a vertex in ${\rm \mathcal{CSR}}(m,n)$. 
For $i \in [n]$, define $g_{i}=\sum _{\alpha=1}^i b_\alpha$. 
If one of the numbers $g_1,g_2,\ldots,g_n$ is zero, 
then $\tau(\mathbf{b})\geq 1+\tau((b_{i+1},\ldots,b_{m}))$. Thus, by using the induction hypothesis we can complete the proof the claim. Now assume that none of the numbers is zero. Noting that these numbers are between 1 and $n-1$, there are two different indexes $i,j$ such that $g_{i}=g_{j}$. Without loss of generality assume that $i<j$. Thus, we have $\sum_ {\alpha =i+1}^j b_\alpha=0$. So, by using the induction hypothesis we can complete the proof the claim.

Combining (\ref{E016}) and (\ref{E016}), shows that 
\begin{equation}\label{E018}
\min\{\tau(\mathbf{b})\mid \mathbf{b}
\in {\rm \mathcal{CSR}}(m,n)\}= 1+ \lfloor  \frac{m-1}{n}  \rfloor.
\end{equation}
This completes the proof.
}
\end{alij}

\subsection{Complexity of calculating the distance between the vertices  of ${\rm \mathcal{CSR}}(m,n)$}

Here, we prove that computing the distance between two vertices of a given  ${\rm \mathcal{CSR}}(m,n)$ is $ \mathbf{NP}$-hard in terms of $n$ and $m$. In more details, we consider the following problem.

{\em Minimum distance problem.}\\
\textsc{Instance}: Two positive integers $m,n$ and two vectors $v,u$ in $\mathbb{Z}^{m}_{n}$ such that for each vector the summation of its coordinates   modulo $n$ is $0$.\\
\textsc{Question}: Compute the distance between $v$ and $u$ in ${\rm \mathcal{CSR}}(m,n)$.

\begin{alijj}{
We reduce 	 3-Partition problem to our problem in polynomial time.
It is shown that  3-Partition problem is $ \mathbf{NP}$-complete in the strong sense \cite{MR1567289}.

{\em 3-Partition.}\\
\textsc{Instance}: A positive integer $ s \in \mathbb{Z}^{+}$ and $3k$ positive integers $a_1, \ldots, a_{3k} \in \mathbb{Z}^{+}$ such that $  s/4 < a_i < s/2$ for each $1 \leq i \leq 3k$ and $ \sum _{i=1}^{3k} a_i=ks$.\\
\textsc{Question}: Can $\{a_1, \ldots, a_{3k}\}$ be partitioned into $k$ disjoint sets $A_1, \cdots , A_k$ such that  for any $i$, $1 \leq i \leq k$, we have $ \sum _{a \in A_i}a= s$?\\	
	
Let $\mathbf{a}=(a_1, \ldots, a_{3k} ,s)$ be an instance of  3-Partition problem. Put $m=3k$ and $ n=s$. Now, consider the vertex $(a_1, \ldots, a_{3k})$ of ${\rm \mathcal{CSR}}(3k,s)$. Note that we have $  s/4 < a_i < s/2$, so $a_i+a_j$ is less than $s$, thus, in any zero partitioning of $\mathbf{a}=(a_1, \ldots, a_{3k} ,s)$ the size of each partition is at least three. Thus,  $\tau(\mathbf{a})\leq k$. We have
$\tau(\mathbf{a})= k$ if and only if  $\{a_1, \ldots, a_{3k}\}$ can be partitioned into $k$ disjoint sets $A_1, \cdots , A_k$ such that  for each $i$, $1 \leq i \leq k$, we have $ \sum _{a \in A_i}a= s$.
Thus, by Lemma \ref{distance}, we have $d(\mathbf{0},\mathbf{a})=2k$ if and only if 3-Partition problem has  a positive answer.	
}\end{alijj}

\subsection{Automorphism Group of ${\rm \mathcal{CSR}} (m,n)$}
 
In \cite{Rook} it is shown that the automorphism group of ${\rm \mathcal{SR}}(m,n)$ is $\mathcal{S}_m$ for $n>3$, 
and is $\mathcal{S}_m \times \mathbb{Z}_2$ for $n=3$,
where $\mathcal{S}_m$ stands for the permutation group on $m$ elements. 
Finally, in this section we study the automorphism group of ${\rm \mathcal{CSR}}(m,n)$. 
Note that the set of vertices of ${\rm \mathcal{CSR}}(m,n)$ is $\{(a_{1},\ldots,a_{m})\in \mathbb{Z}_{n}^{m}\, |\,\sum_{i=1}^{m}a_{i}\equiv 0 \mod n\}$. 
This set is a subgroup of $\mathbb{Z}_{n}^{m}$ and is isomorphic to the group $\mathbb{Z}^{m-1}_{n}$. 
Let $\mathbb{Z}^{\times}_{n}$ be the multiplicative group of integers modulo $n$. 
We show that 
$\mathbf{Aut}( {\rm \mathcal{CSR}}(m,n))\cong \mathcal{S}_m \times  
\mathbb{Z}^{\times}_{n} \times \mathbb{Z}^{m-1}_{n}$ for $n, m>3$.
 
Before we prove our main result we need to present some lemmas.

For every two distinct numbers $a, b \in [m]$ with $a<b$, 
let $S_{ab}=\{\alpha (e_{a}- e_b) \mid \alpha \in  \mathbb{Z}_{n}\}$, 
where $e_i$ is a vector with $m$ coordinates such that all of its coordinates are all zero, except the $i$-th
coordinate that is equal to one. 
For every $a$ and $b$ with $1 \leq a < b \leq m$ define $\mathcal{P}_{ab}= \{\mathbf{\xi}+S_{ab} \mid \mathbf{\xi} \in V({\rm \mathcal{CSR}}(m,n))\}$.
Note that for each $\mathbf{\xi} \in V({\rm \mathcal{CSR}}(m,n))$, the set $\mathbf{\xi}+S_{ab}$
(that is a subset of $\mathcal{P}_{ab}$) contains all vertices with $m-2$ same coordinates (all coordinates except $a$-{th} and $b$-{th} coordinates). Thus, it is easy to see that $\mathcal{P}_{ab}$ is a partitioning of vertices of ${\rm \mathcal{CSR}}(m,n)$.

\begin{lem} \label{11}
	Let $\psi$ be an automorphism of ${\rm \mathcal{CSR}}(m,n)$ and $n,m > 3$. For every $\xi \in V({\rm \mathcal{CSR}}(m,n))$ and $S_{ab}$, 
	we have $\psi(\xi+S_{ab})=\varsigma+S_{pq}$ for a $\varsigma \in V({\rm \mathcal{CSR}}(m,n))$ and some $p, q \in [m]$. 
\end{lem} 

\begin{proof}{
		Any graph isomorphism maps each maximum clique to a maximum clique. 
		So, using  Corollary \ref{cliqueunique},  we can consider three cases.
		\\
		$\bullet$ $n< m$: Let $x, y, z$ be three distinct vertices in $\xi+S_{ab}$.
		These vertices are adjacent  to each others, so  $\psi(x), \psi(y)$ and $\psi(z)$ are adjacent to each others.
		Without loss of generality assume that $\psi(x)-\psi(z)=\alpha (e_{1}-e_{2})$. 
		Thus for  $\psi(x)-\psi(y)$ 
		we have three cases: $\psi(x)-\psi(y)$ is equal to $\beta (e_{1}- e_{2})$ 
		or $\beta (e_{1}- e_{c})$ or $\beta (e_{2}- e_{c})$, where $c\notin \{1,2\}$.
		For the first case, we have $\psi(x), \psi(y), \psi(z) \in \psi(x)+S_{12}$.
		For the second case, we have $\psi(y)-\psi(z)=\alpha (e_{1}- e_{2})-\beta (e_{1}- e_{c})$. Note that $c\neq 2 $ and the vertices  $\psi(y),\psi(z)$ are adjacent.
		Hence $\alpha=\beta$, therefore $\psi(y)-\psi(z)=-\alpha (e_{2}- e_{c})$. So,
		$ \psi(z) - \psi(y) =\alpha (e_{2}- e_{c})$. Consequently, we have $\psi(x) =\xi'+ \alpha e_1$, $\psi(y) =\xi'+ \alpha e_c$ and $\psi(z) =\xi'+ \alpha e_2$, where $\xi' \in \mathbb{Z}^{m}_{n}$.
		Consequently by Corollary \ref{cliqueunique}, $\psi(x), \psi(y)$ and $\psi(z)$ 
		belong to a maximum clique, but $x, y, z$ do not belong to a maximum clique. 
		Thus the second case cannot be happened. 
		The third case is similar to the second one. Hence, $\psi(x),\psi(y)$ 
		and $\psi(z)$ belong to a coset of $S_{pq}$. This completes the proof for the case $n<m$.
		\\		
		$\bullet$ $n>m$: In this case by Corollary \ref{cliqueunique}, 
		every maximum clique has the form $\xi+S_{ab}$, for some $\xi, a$ and $b$.
		Any isomorphism maps a maximum clique to a maximum clique, so $\psi(\xi+S_{ab})$ has the form $\varsigma+S_{pq}$.
		\\		
		$\bullet$ $n=m$: 
		The graph has two types of maximum cliques: 
		The vertices with $m-2$ equal coordinates (coset-type) and the set of vertices $\{u+\beta e_i: i \in [m]\}$ 
		where $u$ is a vertex of the graph and $\beta \in  \mathbb{Z}_n$ (core-type). 
		Note that every two adjacent vertices $x$ and $y=x+ \gamma (e_i - e_j)$ 
		belong to exactly one maximum clique $M$ of coset-type and two of core-types $M^*$ and $M^{**}$. In fact,
		$M=x+\{\alpha(e_i-e_j)\mid \alpha\in \mathbb{Z}_n\}$, $M^{*}=
		x+\gamma e_i+(-\gamma)\{e_i | i \in [m]\}$ and $M^{**}=x-\gamma e_j+(\gamma)\{e_i | i \in [m]\}$. 
		Let $a=x+\alpha(e_i -e_k)$ and $b=x+\beta(e_i -e_k)$, 
		where  $\alpha, \beta \neq \gamma$ and $k \neq i,j$. 
		
		So $\psi$ maps exactly one of $M, M^*, M^{**}$ to a coset-type clique.  
		There is no pair of adjacent vertices  $a$ and $b$ in the graph 
		such that each of $a$ and $b$ has exactly 1 neighbor in each of $M^*$ and $M^{**}$ 
		(\textit{Property I}). But there are adjacent vertices $a$ and $b$ 
		such that each of them has exactly 1 neighbor in each of $M$ and $M^*$. 
		for example let $a=x+\alpha(e_i -e_k)$ 
		and $b=x+\beta(e_i -e_k)$ for distinct numbers $\alpha, \beta \neq \gamma$ 
		and $k \neq i,j$. It is easy to see that $(\psi(M^*), \psi(M^{**}))$ 
		has Property I but $(\psi(M), \psi(M^{*}))$ and $(\psi(M), \psi(M^{**}))$ 
		does not have it. Therefore, $\psi(M)$ is coset-type (and so $\psi(M^*)$ 
		and $\psi(M^{**})$ have another type).
}\end{proof}

\begin{lem}\label{samepart}
	Let $n \geq 4$. If $\psi(\xi+S_{ab})=\varsigma+S_{pq}$ 
	and $\psi(\xi'+S_{ab})=\varsigma'+S_{rs}$, then $\{p, q\} = \{r, s\}$.
\end{lem}

\begin{proof}{
		To the contrary suppose that for two parts of $\mathcal{P}_{ab}$, 
		like $P$ and $P'$ we have  $\psi(P) \in \mathcal{P}_{pq}$ , $\psi(P') \in \mathcal{P}_{rs}$ 
		such that $\{p, q\}\neq \{r, s\}$.
		Since the graph $\mathcal{CSR}(m,n)$ is connected, we can assume that $P$ and $P'$ 
		are adjacent (in other words, every vertex of $P$ has some neighbors in $P'$).
		Now two cases can be considered.
		\\
		$\bullet $  If $\{p, q\} \cap \{r, s\}= \emptyset$, 
		then without loss of generality suppose that $p, q, r, s$ are $1, 2, 3, 4$, respectively. 
		Also assume that  $\psi(P)=(y_1,\ldots, y_m)+S_{12}$ and $\psi(P')=(z_1, \ldots, z_m)+S_{34}$. 
		Since $n \geq 3$, 
		there exists $\mathcal{Y}=(y^*_1, y^*_2, y_3, \ldots, y_m) \in \psi(P)$ 
		such that $y^*_1 \neq z_1$ and $y_2^* \neq z_2$. 
		Every vertex of $P$ has some neighbors in $P'$. So, the vertex $\mathcal{Y}$ has some neighbors in $\psi(P')$. Noting that  $y_1^* \neq z_1$ and $y_2^* \neq z_2$, we conclude that the neighbor of $\mathcal{Y}$ 
		in $\psi(P')$ is $(z_1, z_2, y_3, y_4, y_5, \ldots, y_m)$. 
		But this vertex is in $\psi (P)$. This is a contradiction.
		\\
		$\bullet $ If $\{p, q\} \cap \{r, s\} \neq \emptyset$,
		then without loss of generality assume  that $p=1$, $q=r=2$, $s=3$, 
		$\psi(P)= (y_1, y_2, y_3, ... y_m)+ S_{12}=\{(a, b, y_3, ... y_m) \in V\mid a, b \in \mathbb{Z}_n \}$ 
		and $\psi(P')=(z_1, \ldots, z_m) + S_{23}$. 
		Let $\mathcal{N}$ be the number of integers $i$ such that  $i \geq 4$ and $y_i \neq z_i$. 
		\\	
		$\bullet \bullet  $	If $\mathcal{N} \geq 3$, then $\psi(P)$ and $\psi(P')$ are not adjacent. 
		This is a contradiction.
		\\	
		$\bullet \bullet  $	If $\mathcal{N} =2$, then since $n \geq 2$, there is $w \in \mathbb{ Z}_n$ such that $w \neq z_1$. 
		The vertex $ (w, -w -\sum_{i =3}^{m} y_i, y_3, \ldots, y_m)$  
		is in $\psi (P)$, but does not have any neighbor in $\Psi(P')$. This is a contradiction.
		\\	
		$\bullet \bullet  $	 If $\mathcal{N} = 1$, then since $n \neq 3$, 
		there are $w^*, w^{**} \in \mathbb{Z}_n $ 
		such that $w^*, w^{**} \neq z_1$. 
		Then the neighbor of $(w^*, -w^* - \sum_{i \geq 3} y_i, y_3, \ldots, y_m)$ in $\psi(P')$ 
		should be  $$(z_1, -w^* - \sum_{i =3}^{m} y_i, -z_1 + w^* +\sum_{i \geq 3} y_i - \sum_{i \geq 4} z_i , z_4, \ldots, z_m).$$
		Therefore $y_3 = -z_1 + w^* +\sum_{i \geq 3} y_i - \sum_{i \geq 4} z_i$, 
		and consequently $w^* =z_1 - \sum_{i \geq 4} y_i + \sum_{i \geq 4} z_i$. 
		By the same computation we can reach to the same equality   for $w^{**}$. This leads to a contradiction by  noting that $w^* \neq w^{**}$.
		\\	
		$\bullet \bullet  $ If $\mathcal{N} =0$, then $\psi(P)$ and $\psi(P')$ have a common vertex 
		(that is the vertex $(z_1, -z_1- \sum_{i \geq 3} y_i , y_3, y_4, \ldots, y_m)$). This is a contradiction.
}\end{proof}

\begin{lem}\label{permutcoord}
	There is a permutation $\sigma$ over $\mathbb{Z}_n$
	such that for every $a$ and $b$, $\psi (\mathcal{P}_{ab})= \mathcal{P}_{\sigma(a)\sigma(b)}$.
\end{lem}

\begin{proof}{
		For a partitioning $\mathcal{P}$, 
		we denote the image of the parts of $\mathcal{P}$ under the automorphism $\psi$ by $\psi(\mathcal{P})$.
		By Lemma \ref{samepart}, 
		we know that for every $a, b, c \in \mathbb{Z}_n$, 
		$\psi (\mathcal{P}_{ab})= \mathcal{P}_{\alpha \beta}$ and $\psi(\mathcal{P}_{ac})=\mathcal{P}_{\gamma\delta}$ 
		for some $\alpha, \beta, \gamma, \delta$. 
		
		Since we assumed that $n>3$, in order
		to prove  the lemma it is enough to show that  $\{\alpha, \beta\}\cap \{\gamma, \delta\} \neq \emptyset$ (Property 1). 
		Note that if Property 1 is true, then for every $a, b_1, \ldots, b_i$ there are $\alpha, \beta_1, \ldots, \beta_i$ such that
		$\psi (\mathcal{P}_{ab_1}), \psi (\mathcal{P}_{ab_2}), \ldots, \psi (\mathcal{P}_{ab_i})$ are
		$\mathcal{P}_{\alpha \beta_1}, \mathcal{P}_{\alpha \beta_2}, \ldots, \mathcal{P}_{\alpha \beta_i}$, respectively.

		We prove Property 1 by contradiction. In order to reach to a contradiction we prove two opposite things.\\
	\\
	{\bf Fact 1.}		If  $\{\alpha, \beta\}\cap \{\gamma, \delta\} = \emptyset$, then there are parts $S \in \mathcal{P}_{\alpha \beta}$ and $S' \in \mathcal{P}_{\gamma\delta}$ such that the number of edges from $S$ to $S'$ is exactly $4$.
	\\
	{\bf Proof of Fact 1.}		
	Note  that $S_{\alpha \beta}$ is $\{c(e_\alpha -e_\beta): c \in \mathbb{Z}_n\}$. 
	For every distinct $\alpha, \beta, \gamma, \delta \in \mathbb{Z}_n$, 
	there are two parts $S=\mathbf{0}+ S_{\alpha \beta}= \{c(e_\alpha - e_\beta)\}$ 
	and $S'=(e_{\alpha}-e_{\gamma}) + S_{\gamma \delta}=\{e_{\alpha}+ (c'-1)e_{\gamma}-c'e_\delta)\}$, 
	such that exactly 4 edges of $\mathcal{CSR}(m,n)$ have one end-vertex in $S$ and another end-vertex in $S'$. 
	In fact, one end-vertex of  an edge is in $S$ and the other end-vertex is in $S'$ if and only if $c=0, 1$ and $c'=0, 1$. 
	For example if $\alpha =1, \beta=2 , \gamma=3$ and $\delta=4$, 
	those 4 edges are $\{(0^m),(1, 0, -1, 0^{m-3})\}$, $\{(0^m),(1, 0, 0, -1, 0^{m-4})\}$, 
	$\{(1, -1, 0^{m-2}),(1, 0, -1, 0^{m-3})\}$ and $\{(1, -1, 0^{m-2}),(1, 0, 0, -1, 0^{m-4})\}$.   $ \clubsuit$
	\\
	\\
{\bf Fact 2.} There are no parts $S \in \mathcal{P}_{ab}$ and $S' \in \mathcal{P}_{ac}$ such that the number of edges between $S$ and $S'$ is exactly $4$.
\\
{\bf Proof of Fact 2.}
Without loss of generality suppose that $a=1, b=2$ and $c=3$.
		Let $S \in \mathcal{P}_{12}$ and $S' \in \mathcal{P}_{13}$ be two parts.
		So there are $x_3, x_4, x_5, \ldots, x_m \in \mathbb{Z}_n$ and $y_2, y_4, y_5,\ldots, y_m \in \mathbb{Z}_n$ such that
		$S=\{(a, b, x_3, x_4, x_5, \ldots, x_m) \in  V \mid a, b \in \mathbb{Z}_n\}$ 
		and $S'=\{(a, y_2, b, y_4, y_5, \ldots, y_m) \in  V \mid a, b \in \mathbb{Z}_n \}$.
		Denote the number of edges between $S$ and $S'$ by $\mathcal{E}$.
		Define $X=|\{i \mid 3 \leq i \leq m, x_i \neq y_i\}|$.
		By considering the possible values of $X$, we show that $\mathcal{E} \neq 4$.
\\	
		$\bullet $ $X \geq 3$: In this case  by the definition of edges of ${\rm \mathcal{CSR}} (m,n)$, we have $\mathcal{E}=0$.
\\		
		$\bullet $ $X=2$: In this case if there is an edge $e=(x_1, \ldots, x_m)(y_1, \ldots, y_m)$ from $S$ to $S'$, then we have $x_1=y_1$, $x_2=y_2$ and $x_3=y_3$.
		So, by the definition of the graph ${\rm \mathcal{CSR}} (m,n)$ there is at most one such edge. In other words, $\mathcal{E} \leq 1$.
\\		
		$\bullet$ $X=1$: In this case if there is an edge $e=(x_1,\ldots, x_m)(y_1, \ldots, y_m)$ from $S$ to $S'$, then for each $i \in \{1, 2, 3\}$, that edge is a unique edge
 such that $x_i \neq y_i$.
Thus $\mathcal{E}=3$.
\\		
		$\bullet $ $X=0$: Similar to the previous case, there are exactly two $i,j \in \{1, 2, 3\}$ such that $x_i \neq y_i$ and $x_j \neq y_j$.
		For $x_2=y_2$ there are exactly $n-1$ edges from $S$ to $S'$.
		Also for $x_3=y_3$ there are exactly $n-1$ edges from $S$ to $S'$.
		So, we have $\mathcal{E} \geq 2n-2$. Note that $n>3$, so we have $\mathcal{E}>4$.  $ \clubsuit$

Having both Facts 1, and 2, we reach to a contradiction. This completes the proof of lemma.
	}
\end{proof}

\begin{lem}\label{3.12}
	There are $m$ permutations $\psi_1, \ldots, \psi_m$ over $\mathbb{Z}_n$ 
	and a permutation $\sigma$ on $[m]$ such that $\psi( \sum_i x_{i}e_{i})= \sum_i \psi_i(x_{i})e_{\sigma(i)}$. 
	In other words, the restriction of $\psi$ on each coordinate is a permutation over $\mathbb{Z}_n$.
\end{lem}

\begin{proof}{
		Let $U$ and $U'$ be two vertices of ${\rm \mathcal{CSR}}(m,n)$
		with a same value in the $j$-th coordinate.
		There is a path $U=U_0, U_1, \ldots, U_k=U'$
		such that all vertices of it have the same $j^{th}$ coordinate.
		Then for every $0 < i \leq k$, there exist $a_{i}, b_{i} \neq j$ 
		such that $U_i- U_{i-1}$ belongs to $S_{a_{i}b_{i}}$. 
		By Lemma \ref{permutcoord} there is a permutation $\sigma$ 
		such that $\psi (\mathcal{P}_{a_ib_i})= \mathcal{P}_{\sigma(a_i)\sigma(b_i)}$.
		Since $a_{i}, b_{i} \neq j$, all vertices in  $\mathcal{P}_{\sigma(a_i)\sigma(b_i)}$
		have a same $\sigma (j)$-th coordinate.
		Therefore $\sigma (j)$-th coordinate of all $\psi(U_i)$'s are same. 
		Also, $\sigma (j)$-th coordinate of $\psi(U)$ and $\psi(U')$ are same. 
		Since $j$, $U$ and $U'$ are selected arbitrarily, the proof is completed.
}\end{proof}

\begin{lem} \label{55}
	There exists $d_1, \ldots, d_m \in \mathbb{Z}_n$ 
	such that $\psi_1 + d_1 \equiv \ldots \equiv \psi_m +d_m$.
\end{lem}

\begin{proof}{
		Let $x,y$ be distinct members of $\mathbb{Z}_n$. Let $V=(x,y,-x-y,0^{m-3})$ 
		and $U=(y,x,-x-y,0^{m-3})$. Hence $V$ and $U$ are adjacent. Thus, by Lemma \ref{3.12},
		\begin{equation} \label{xc}
		\psi_1(x)+\psi_2(y) + \psi_3(-x-y)+ \sum^m_{i=4} \psi_i(0)=\psi_1(y)+\psi_2(x) + \psi_3(-x-y)+ \sum^m_{i=4} \psi_i(0).
		\end{equation}
		So, $\psi_1(x)+\psi_2(y)=\psi_1(y)+\psi_2(x)$. Noting that  $x$ and $y$ are selected arbitrarily, so
		there exists $c_{2}$ such that $\psi_1\equiv \psi_2 + c_2 $. 
		Similarly, there exists $c_{i}$ such that $\psi_1\equiv \psi_i + c_i $, for every $i$.
}\end{proof}

\begin{lem} \label{x2}
	There are $c, d \in \mathbb{Z}_n$ 
	such that $c$ and $n$ are coprime and for any $x$, $\psi_1(x)= cx+d$.
\end{lem}

\begin{proof}{
		Consider 4 vertices $A=(x, -x, 0^{m-2}), B=(x-1, -x+1, 0^{m-2} ), 
		C=(x, -x+1, -1, 0^{m-3}), D=(x+1, -x, -1, 0^{m-3})$.
		The vertices $A$ and $B$ are adjacent, so by (\ref{xc}) and Lemma \ref{55},
		\begin{equation} \label{11111}
		\psi_1(x)+\psi_2(-x)=\psi_1(x-1)+\psi_2(-x+1).
		\end{equation} 
		Also, $C$ and $D$ are adjacent. Thus,
		\begin{equation} \label{222}
		\psi_1(x)+\psi_2(-x+1)=\psi_1(x+1)+\psi_2(-x).
		\end{equation}
		So by (\ref{11111}) and (\ref{222}), for any $x$ 
		we have $$\psi_1(x+1)-\psi_1(x)=\psi_1(x)-\psi_1(x-1),$$ 
		which completes the proof.
}\end{proof}

\begin{lem} \label{8888}
	Every bijective function $\psi$ over the vertices of 
	${\rm \mathcal{CSR}}(m,n)$ is an automorphism 
	if and only if there are a permutation $\sigma$ 
	over $[m]$, $c, d_1, \ldots, d_m \in \mathbb{Z}_n$ 
	such that $c$ and $n$ are two coprime numbers, $\sum d_i =0$ 
	and for every $x=(x_1, \ldots, x_m)$, 
	$\psi(x)= (c x_{\sigma(1)} +d_1, \ldots, c x_{\sigma(m)} +d_m)$.
\end{lem}

\begin{proof}{
Every bijective function with mentioned  properties  is an automorphism. 
		The other direction of the lemma can be proved by using Lemmas \ref{3.12}, \ref{55} and \ref{x2}. 
}\end{proof}

Now, we are ready to prove the main result.

\begin{alijjj}{
Every automorphism is determined by  a permutation $\sigma$ 
over $[m]$ and $c, d_1, \ldots, d_m \in \mathbb{Z}_n$ 
such that $c$ and $n$ are two coprime numbers and $\sum d_i =0$. 
Note that every such $\sigma, c, d_1, \ldots, d_m$ determine an automorphism.

We should note that $\sigma$ and $c$ and $d_1, ..., d_{m-1}$ are independent from each others. But by Lemma \ref{8888}, $d_m=-\sum^{m-1}_{i=1} d_i$.
The permutation $\sigma$ over $[m]$ is selected from the permutation group $S_m$. 
Also, $c \in \mathbb{Z}_n$ is coprime to $n$. Hence, $c$ is selected from the multiplicative group $\mathbb{Z}^{\times}_{n}$.
Each of $d_1, ..., d_{m-1}$ is selected from $\mathbb{Z}_{n}$.
Thus,  $\mathbf{Aut}({\rm \mathcal{CSR}}(m,n))\cong \mathcal{S}_m \times \mathbb{Z}^{\times}_{n} \times \mathbb{Z}^{m-1}_{n}$. 	
}\end{alijjj}

\section{Conclusion}

In this work, we investigated the properties of simplicial rook graphs and  cyclic simplicial rook graphs. 
We calculated the independence number of ${\rm \mathcal{SR}}(m,n)$ that is the  maximum number of  non attacking rooks which can be placed
on a $(m-1)$-dimensional simplicial chessboard of side length $n+1$. We proved that $\alpha({\rm \mathcal{SR}}(m,n))=\big(1-o(1)\big)\frac{\binom{n+m-1}{n}}{m}$ and  $\gamma({\rm \mathcal{SR}}(m, n))= \Theta (n^{m-2})$. 
Also,  we determined several properties of cyclic simplicial rook graphs such as independence number, chromatic number and automorphism group. Among other results, we also proved that computing the distance between two vertices of a given  ${\rm \mathcal{CSR}}(m,n)$ is NP-hard in terms of $n$ and $m$.

\bibliographystyle{plain}
\bibliography{POref}

\end{document}